\documentclass[11pt,reqno]{amsart}
\usepackage{amsfonts}
\usepackage{ifthen}
\usepackage{amsthm}
\usepackage{amsmath}
\usepackage{graphicx}
\usepackage{caption}
\usepackage{subcaption}
\usepackage{amscd,amssymb}
\usepackage{color}
\usepackage{hyperref}
\usepackage{array,multirow,makecell}

\usepackage{cite,epstopdf}

\setcellgapes{1pt}
\makegapedcells
\newcolumntype{R}[1]{>{\raggedleft\arraybackslash }b{#1}}
\newcolumntype{L}[1]{>{\raggedright\arraybackslash }b{#1}}
\newcolumntype{C}[1]{>{\centering\arraybackslash }b{#1}}
\newcounter{minutes}
\divide\time by 60
\newcounter{hours}
\multiply\time by 60 \addtocounter{minutes}{-\time}

\newtheorem{theorem}{Theorem}
\newtheorem{lemma}{Lemma}

\newtheorem{corollary}{Corollary}
\newtheorem{remark}{Remark}

\numberwithin{equation}{section}

\title[Certain geometric properties of  planar harmonic mapping involving four parameter Wright functions]{Certain geometric properties of  planar harmonic mapping involving four parameter Wright functions}

\author[A. Kumar ]
{Anish Kumar}

\address{{\bf Anish Kumar}\newline
Department of Mathematics,
Dr. Shyama Prasad Mukherjee Unuversity,\newline
Ranchi 834008, Jharkhand, India}
\email{ak8107690@gmail.com}

\keywords{Analytic functions, univalent functions, harmonic functions, Mittag-Leffler functions, harmonic starlike, harmonic convex}
\subjclass[2020]{31B20; 30C45; 30CI5; 33E12; }
\begin{document}

\begin{abstract}
The primary aim of this paper is to construct harmonic mapping associated with four parameter Wright functions. Further sufficient conditions have been established so that this harmonic mapping satisfies geometric properties such as harmonic starlikeness and convexity. Moreover, we examine relation between several subclasses of family of harmonic close-to-convexity mappings. The results obtained in this paper are novel and their significance is illustrated by several consequences.
\end{abstract}
\maketitle

\section{Introduction and Motivation}
Harmonic functions play a vital role in several problems in applied mathematics and are also important because of their use in the minimal surface. Many different geometers, such as Lewy and Rado \cite{lewy}, Choquest \cite{choquet} and  Kneser\cite{kneser} studied the harmonic functions. A continuous complex value mapping $f = u +iv$ is said to be harmonic in a domain $D$ connected simply in the complex plane if it holds $f_{z\Bar{z}}=0$ in $D$, that is, $u$ and $v$ are real harmonic functions in $D$. This type of harmonic function $f$ can be represented as $f = h + \Bar{g}$, where $h$ and $g$ are analytic in $D$ with $h(0)=0$, $h^{\prime}(0)=1$, $g(0)=0$ and have Taylor series representations of the form 
\begin{align}\label{eqharmonic}
h(z)=z+\sum_{n=2}^{\infty}a_{n}z^{n}, \quad g(z)=\sum_{n=1}^{\infty}b_{n}z^{n}, \quad |b_{1}|<1.
\end{align}
We often say $h$ and $g$ the analytic and co-analytic parts of $f$ respectively. This above class of harmonic mapping is denoted as $\mathaccent"705E{H}$. A necessary and sufficient condition for $f$ in the class $\mathaccent"705E{H}$ to be sense preserving and locally univalent is $|h'(z)|>|g'(z)|$ in $D$. The subclass of $\mathaccent"705E{H}$ consisting of functions that are sense-preserving and univalent in $D$ are denoted by $HS$. For $0\leq \alpha <1$,\\
$$HS^{*}(\alpha)=\left\{f\in \mathaccent"705E{H}: \frac{\partial}{\partial \theta}(\arg(f(re^{\i\theta})))>\alpha \right\}$$.
$$HK(\alpha)=\left\{f\in \mathaccent"705E{H}: \frac{\partial}{\partial \theta}(\arg(\frac{\partial}{ \partial\theta}(f(re^{i\theta}))))>\alpha \right\},$$
for $z=re^{i\theta}$, $0\leq \theta \leq 2\pi$ and $0\leq r <1$. It can be noted that $HS^{*}=HS^{*}(0)$ is the class of harmonic starlike functions with respect to origin in $D$. Similarly $HK=HK(0)$ is the class of harmonic convex function in $D$. For $0\leq \alpha <1$ harmonic starlike of order $\alpha$ and harmonic convex of order $\alpha$ denoted by $HS^{*}(\alpha)$ and $HK(\alpha)$, respectively.

The subclass of $\mathaccent"705E{H}$ denoted by $S\mathaccent"705E{H}$ that is sense-preserving and univalent in $D$. It can be seen that $\frac{f-\Bar{B_{1}f}}{1-|B_{1}|^{2}}  \in S\mathaccent"705E{H}$, whenever $f\in S\mathaccent"705E{H}$. There is a subclass $S\mathaccent"705E{H^{o}}$ taken into account by $S\mathaccent"705E{H}$ given by \\
$$S\mathaccent"705E{H^{o}}=\left\{f= h + \Bar g \in S\mathaccent"705E{H}: g'(0)=B_{1}=0 \right\}.$$
The subclass $S\mathaccent"705E{H^{o}}$ and $S\mathaccent"705E{H}$ were studied in \cite{clunie}. It can be noted that $HS(0)=S\mathaccent"705E{H}$. The class $C$ of all close-to-convex analytic functions in $D$, assume that $C\mathaccent"705E{H^{o}}$ and $C\mathaccent"705E{H}$ are the subclass respectively of $S\mathaccent"705E{H^{o}}$ and $S\mathaccent"705E{H}$ such that for any $f \in C\mathaccent"705E{H^{o}}$ and $f \in C\mathaccent"705E{H}$, $f(D)$ is close to convex in $D$. It can be recall that a domain $D$ is close to convex if the complement of $D$ can be written as a union of nonincreasing half lines. 

A complex valued harmonic function $f$ in $D$ is called typically real given that $f(z)$ is real iff $z$ is real. All sense preserving typically real harmonic function $f \in \mathaccent"705E{H}$ are denoted by the class $T\mathaccent"705E{H}$ and with $g(0)=0$ is denoted by $T\mathaccent"705E{H^{o}}$. Suppose that $SRH(\alpha)$ and $KRH(\alpha)$ denote the subclass of $HS^{*}(\alpha)$ and $HK(\alpha)$ respectively containing the function $f=h+\Bar g$, where  
\begin{equation}\label{eq1}
  h(z)=z-\sum_{n=2}^{\infty}|A_{n}|z^{n}, \quad g(z)=\sum_{n=1}^{\infty}|B_{n}|z^{n}, \quad |B_{1}|<1.  
\end{equation}
Assume that $SRH=SRH(0),$ $KRH=KRH(0).$ 

The Wright function
$$\mathcal{W}_{\alpha, \beta}(z)=\sum_{n=0}^{\infty}\frac{z^{n}}{\Gamma(\alpha n +\beta)n!},\quad \beta,z \in C, \alpha >-1,$$ was firstly introduced by E. M. Wright \cite{wright} in connection with the theory of partitions. It plays important role in fractional calculus \cite{luchko,podlubny}, integral equation of Hankel type, Mikusiski operational calculus and  stochastic process. For a history regarding Wright function and applications, one can refer \cite{mainardi}. It is noted that Wright functions is an entire function of order $\frac{1}{1+\alpha}$, which is also generalized Bessel functions \cite{podlubny,baricz}. These functions generalized hypergeometric functions \cite{stegun, Andrews}.\\
The four parameter Wright functions \cite{luchko}\\
$$\mathcal{W}_{{(\alpha, \beta),(\gamma, \delta)}}(z)=\sum_{n=0}^{\infty}\frac{z^{n}}{\Gamma(\alpha  +n\beta)\Gamma(\gamma  +n \delta)},\quad \alpha,\gamma,z \in C, \beta,\delta \in R,$$ was studied by E. M. Wright for $\beta,\delta>0$ \cite{Wright}. In addition he derived many properties of $\mathcal{W}_{(\alpha, \beta),(\gamma, \delta)}(z)$ for $\gamma=\delta=1$ and $-1<\beta<0$ \cite{right}.  It can be justified \cite{mehrez} that if $\beta + \delta>0$, then the infinite series expansion of $\mathcal{W}_{{(\alpha, \beta)},({\gamma, \delta})}(z)$ converges absolutely for all $z\in C$. It is well known that \cite{luchko} $\mathcal{W}_{{(\alpha, \beta)},({\gamma, \delta})}(z)$ is an entire function for $\alpha, \gamma \in C$ and $0<\beta < \delta.$\\
It can be noted that 
$$W_{{(\alpha, \beta)},({\gamma, \delta})}(z)\\=z\Gamma(\alpha)\Gamma(\gamma)\mathcal{W}_{{(\alpha, \beta)},({\gamma, \delta})}(z)=z+\sum_{n=2}^{\infty}\frac{z^{n}}{\Gamma(\alpha  +(n-1)\beta)\Gamma(\gamma  +(n-1) \delta)},$$ where $ \alpha, \gamma \in C, \quad \forall \quad \beta, \delta>-1$ belongs to normalized class for all $z \in D$. Some geometric properties of four-parameter Wright functions $W_{{(\alpha, \beta)},({\gamma, \delta})}(z)$ have beed studied in \cite{das}.\\
In this paper, we shall study for the real valued $\beta, \delta$ and $z\in D$ and now define a linear operator by convolution as follwos for the functions:\\
$$W_{{(\alpha_{1}, \beta_{1})},({\gamma_{1}, \delta_{1}})}(z)=z+\sum_{n=2}^{\infty}\frac{\Gamma(\alpha_{1})\Gamma(\gamma_{1})A_{n}z^{n}}{\Gamma(\alpha_{1}  +(n-1)\beta_{1})\Gamma(\gamma_{1}  +(n-1) \delta_{1})}$$ and \\ $$W_{{(\alpha_{2}, \beta_{2})},({\gamma_{2}, \delta_{2}})}(z)=z+\sum_{n=2}^{\infty}\frac{\Gamma(\alpha_{2})\Gamma(\gamma_{2})B_{n}z^{n}}{\Gamma(\alpha_{2}  +(n-1)\beta_{2})\Gamma(\gamma_{2}  +(n-1) \delta_{2})}.$$
Now define a linear convolution operator corresponding to $f=h+\bar{g}$\\
$L:\hat{H} \rightarrow \hat{H}$ as: \\
\begin{align*}
    L(f)&=f \ast(W_{{(\alpha_{1}, \beta_{1})},({\gamma_{1}, \delta_{1}})}(z)+\bar {\sigma W_{{(\alpha_{2}, \beta_{2})},({\gamma_{2}, \delta_{2}})}(z)})\nonumber\\
    &=H(z) + \bar{\sigma G(z)}=h \ast{W_{{(\alpha_{1}, \beta_{1})},({\gamma_{1}, \delta_{1}})}(z)}+\bar {\sigma   g * W_{{(\alpha_{2}, \beta_{2})},({\gamma_{2}, \delta_{2}})}(z)}, \quad |\sigma|<1,
\end{align*}
where 
\begin{align}\label{eqHarmonic}
H(z)=z+\sum_{n=2}^{\infty}\frac{\Gamma(\alpha_{1})\Gamma(\gamma_{1})A_{n}z^{n}}{\Gamma(\alpha_{1}  +(n-1)\beta_{1})\Gamma(\gamma_{1}  +(n-1) \delta_{1})},\nonumber\\
G(z)=\sum_{n=1}^{\infty}\frac{\Gamma(\alpha_{2})\Gamma(\gamma_{2})B_{n}z^{n}}{\Gamma(\alpha_{2}  +(n-1)\beta_{2})\Gamma(\gamma_{2}  +(n-1) \delta_{2})}.  
\end{align}
Since four-parameter Wright function $\mathcal{W}_{{(\alpha, \beta),(\gamma, \delta)}}(z)$ is a generalization of Wright function $\mathbb{W}_{(\alpha, \beta)}(z)=\mathcal{W}_{{(\alpha, \beta),(1,1)}}(z)$, similar to above linear operator we can define :

$L:\hat{H} \rightarrow \hat{H}$ as: \\
\begin{align}\label{eqwright}
    L(f)&=f \ast(W_{{(\alpha_{1}, \beta_{1})},(1,1)}(z)+\bar {\sigma W_{{(\alpha_{2}, \beta_{2})},(1,1)}(z)})\nonumber\\
    &=H(z) + \bar{\sigma G(z)}=h \ast{W_{{(\alpha_{1}, \beta_{1})},(1,1)}(z)}+\bar {\sigma   g * W_{{(\alpha_{2}, \beta_{2})},(1,1)}(z)}, \quad |\sigma|<1,
\end{align}
where 
\begin{align}
H'z)=z+\sum_{n=2}^{\infty}\frac{\Gamma(\alpha_{1})\Gamma(\gamma_{1})A_{n}z^{n}}{\Gamma(\alpha_{1}  +(n-1)\beta_{1})},\nonumber\\
G'(z)=\sum_{n=1}^{\infty}\frac{\Gamma(\alpha_{2})\Gamma(\gamma_{2})B_{n}z^{n}}{\Gamma(\alpha_{2}  +(n-1)\beta_{2})}.  
\end{align}

In \cite{op, sarkar, aghalary}, researchers derived several require conditions for harmonic mapping asscoiated with special functions hypergeometric function, Bessel function, Mittag-leffler functon, Wright-functions lies on geometric characteristics such as Harmonic convexity and starlikeness. Ahuja at.al \cite{ahuja} studied harmonic close -to-convexity for harmonic mapping involving hypergeometric function. In \cite{sahu}, connections between various subclasses of hamonic mapping have been discussed related to Wright function.
Motivated by above mentioned result, we have decided to study geometric properties such as harmonic convexity , starlikeness, close-to-convexity for harmonic mapping involving four-parameter Wright function in this paper.

The paper is systemized like this. In Section \ref{sec2}, we have mentioned some crucial lemmas, which will turn out to be useful to derive our iportant results. Section \ref{sec3} discusses some sufficient conditions in such a way that the harmonic mapping involving four-parameter Wright function satisfies like harmonic starlikeness. In Section \ref{sec4}, we have considered normalized harmonic mapping involving four-parameter Wright function and harmonic. Harmonic close-to-convexity has been established for the above function in section \ref{sec5}. In the section \ref{sec6}, several consequences and conclusion of main results have been investigated

In order to prove the main findings, we require following lemmas:

\section{Lemma}\label{sec2}
\begin{lemma}\cite{jahangiri}\label{lem1}
    Let $f= h + \bar g \in \hat{H}$ with $h$ and $g$ of the form \eqref{eq1}. If for some $\alpha(0\leq \alpha <1),$ the inequality $$\sum_{n=2}^{\infty}(n-\alpha)|A_{n}|+\sum_{n=1}^{\infty}(n+\alpha)|B_{n}|\leq 1-\alpha$$ is satisfied, then f is harmonic, sense preserving, univalent in $D$ and $f\in HS^{*}(\alpha)$.
\end{lemma}
\begin{lemma}\cite{karlsson}\label{lem2}
  Let $f= h + \bar g \in \hat{H}$ where $h$ and $g$ are given by \eqref{eq1}. If for $\alpha(0\leq \alpha <1),$ the inequality $$\sum_{n=2}^{\infty}(n(n-\alpha))|A_{n}|+\sum_{n=1}^{\infty}(n(n+\alpha))|B_{n}|\leq 1-\alpha$$ is satisfied, then f is harmonic, sense preserving, univalent in $D$ and $f\in HK(\alpha)$.  
\end{lemma}
\begin{lemma}\label{lem3}
    Suppose $h$ and $g$ are analytic function in $D$ with $|g^{\prime}(0)|<|h^{\prime}(0)|$ and that $h+\epsilon g \in C$ for each $\epsilon(|\epsilon|=1$. Then $h + \bar g \in C_{\hat{H}}.$
\end{lemma}

\begin{lemma}\label{lem5}
    If $q(z)=z+\sum_{n=2}^{\infty}t_{n}z^{n}$ is analytic in $D$, then q maps onto a starlike domain if $\sum_{n=2}^{\infty}n|t_{n}|\leq 1.$
\end{lemma}

\begin{lemma}\label{lem4}
    If $f=h + \bar g \in K_{\hat{H}^{o}}$ is given by \eqref{eq1}, then $|A_{n}|\leq (n+1)/2$, $|B_{n}|\leq (n-1)/2$, $n=1,2,3,\cdots$.
\end{lemma}

The results in the next two lemmas found in \cite{wang} and \cite{small}, respectively.
\begin{lemma}\label{lem7}
    If $f=h + \bar g \in C_{\hat{H}^{o}}(S_{\hat{H}^{o}} or T_{\hat{H}^{o}})$ is given by \eqref{eq1}, then 
    \begin{align*}
    |A_{n}|\leq \frac{(2n+1)(n+1)}{6}, |B_{n}|\leq \frac{(2n-1)(n-1)}{6}, n=1,2,3,\cdots.
     \end{align*}
\end{lemma}

\begin{lemma}\label{lemclosetoconvex}
    If $f=h + \bar g \in C_{\hat{H} }$ is given by \eqref{eq1}, then 
    \begin{align*}
    |A_{n}|\leq \frac{(2n+1)(n+1)}{6}+ \frac{(2n-1)(n-1)}{6}|B_{1}|, \quad n=1,2,3,\cdots\\
     |B_{n}|\leq \frac{(2n-1)(n-1)}{6}+ \frac{(2n+1)(n+1)}{6}|B_{1}| \quad n=1,2,3,\cdots.
     \end{align*}
\end{lemma}

 The result in next lemma from \cite{jahangiri, jmj} provide necessary and sufficient conditions in terms of coefficient for the functions in $SRH(\alpha)$ and $KRH(\alpha)$.
\begin{lemma}\label{lem6}
    Let $f= h + \bar g \in \hat{H}$ where $h$ and $g$ are given by \eqref{eq1}.  Then   $$f \in SRH(\alpha) \Leftrightarrow \sum_{n=2}^{\infty}((n-\alpha))|A_{n}|+\sum_{n=1}^{\infty}((n+\alpha))|B_{n}|\leq 1-\alpha,$$
 $$f \in KRH(\alpha) \Leftrightarrow
   \sum_{n=2}^{\infty}(n(n-\alpha))|A_{n}|+\sum_{n=1}^{\infty}(n(n+\alpha))|A_{n}|\leq 1-\alpha.$$
\end{lemma}

We shall also use the following lemma, through out the paper.
\begin{lemma}
  Assume that $\beta, \delta \geq 0$, $\alpha, \gamma> 0$.\\
       {\rm (i)}$\sum_{n=1}^{\infty}\left|\frac{n(n+1)\Gamma(\alpha)\Gamma(\gamma)}{\Gamma(\alpha  +n \beta)\Gamma(\gamma  +n \delta)}\right|=W_{{(\alpha, \beta)},{(\gamma, \delta)}}^{\prime\prime}(1)$.\\
      {\rm (ii)}$\sum_{n=1}^{\infty}\left|\frac{(n+1)\Gamma(\alpha)\Gamma(\gamma)}{\Gamma(\alpha  +n \beta)\Gamma(\gamma  +n \delta)}\right|=W_{{(\alpha, \beta)},({\gamma, \delta})}^{\prime}(1)-1$.\\ {\rm (iii)}$\sum_{n=0}^{\infty}\left|\frac{\Gamma(\alpha)\Gamma(\gamma)}{\Gamma(\alpha  +n \beta)\Gamma(\gamma  +n \delta)}\right|=W_{{(\alpha, \beta)},({\gamma, \delta})}(1)$.\\
       {\rm (iv)}$\sum_{n=0}^{\infty}\left|\frac{n(n+1)(n-1)\Gamma(\alpha)\Gamma(\gamma)}{\Gamma(\alpha  +n \beta)\Gamma(\gamma  +n \delta)}\right|=W_{{(\alpha, \beta)},{(\gamma, \delta)}}^{\prime\prime\prime}(1)$.
\end{lemma}

In the next section, we discuss some sufficient conditions in such a way that the harmonic mapping involving four-parameter Wright function satisfies harmonic starlikeness.

\section{Harmonic Starlikeness}\label{sec3}
\begin{theorem}
    Suppose that $\beta_{1},\beta_{2}, \delta_{1}, \delta_{2}\geq 0$, $\alpha_{1},\alpha_{2}, \gamma_{1}, \gamma_{2}> 0$ and $|\sigma|<1$. Also assume $f= h + \bar g \in \hat{H}$ if the inequality $|A_{n}|\leq 1, |B_{n}|\leq 1$ and \\$$ (W_{{(\alpha_{1}, \beta_{1})},({\gamma_{1}, \delta_{1}})}^{\prime}(1)-1) -\alpha(W_{{(\alpha_{1}, \beta_{1})},({\gamma_{1}, \delta_{1}})}(1)-1)+|\sigma|(W_{{(\alpha_{2}, \beta_{2})},({\gamma_{2}, \delta_2})}^{\prime}(1) + \alpha W_{{(\alpha_{2}, \beta_{2})},({\gamma_{2}, \delta_{2}})}(1))\leq (1-\alpha)$$ are satisfied then $L(f)$ is harmonic starlike function, i.e $L(f)\in HS^{*}(\alpha)$.
\end{theorem}

\begin{proof}
Let $f= h + \bar g \in \hat{H}$, where $h$ and $g$ given by the equation \eqref{eqharmonic}. We have to prove $L(f)\in HS^{*}(\alpha)$, where $$L(f)=z+\sum_{n=2}^{\infty}\frac{\Gamma(\alpha_{1})\Gamma(\gamma_{1})A_{n}z^{n}}{\Gamma(\alpha_{1}  +(n-1)\beta_{1})\Gamma(\gamma_{1}  +(n-1) \delta_{1})}+\sum_{n=1}^{\infty}\frac{\Gamma(\alpha_{2})\Gamma(\gamma_{2})B_{n}z^{n}}{\Gamma(\alpha_{2}  +(n-1)\beta_{2})\Gamma(\gamma_{2}  +(n-1) \delta_{2})}.$$
    In view of lemma \ref{lem1},
it suffices to show that \begin{align*}
        S_{2}=&\sum_{n=2}^{\infty}(n-\alpha)\left|\frac{\Gamma(\alpha_{1})\Gamma(\gamma_{1})}{\Gamma(\alpha_{1}  +(n-1)\beta_{1})\Gamma(\gamma_{1}  +(n-1) \delta_{1})}\right||A_{n}|\\
        &+|\sigma|\sum_{n=1}^{\infty}(n+\alpha)\left|\frac{\Gamma(\alpha_{2})\Gamma(\gamma_{2})}{\Gamma(\alpha_{2}  +(n-1)\beta_{2})\Gamma(\gamma_{2}  +(n-1) \delta_{2})}\right||B_{n}|\leq (1-\alpha).
        \end{align*}
        Now, we have
        \begin{align*}
        S_{2} &  \leq \sum_{n=1}^{\infty}(n+1-\alpha)\left|\frac{\Gamma(\alpha_{1})\Gamma(\gamma_{1})}{\Gamma(\alpha_{1}  +n \beta_{1})\Gamma(\gamma_{1}  +n \delta_{1})}\right|\\
        & + |\sigma|\sum_{n=0}^{\infty}(n+1+\alpha)\left|\frac{\Gamma(\alpha_{2})\Gamma(\gamma_{2})}{\Gamma(\alpha_{2}  +n\beta_{2})\Gamma(\gamma_{2}  +n \delta_{2})}\right|=\\
        &  (W_{{(\alpha_{1}, \beta_{1})},({\gamma_{1}, \delta_{1}})}^{\prime}(1)-1) -\alpha(W_{{(\alpha_{1}, \beta_{1})},({\gamma_{1}, \delta_{1}})}(1)-1)+|\sigma|(W_{{(\alpha_{2}, \beta_{2})},({\gamma_{2}, \delta_2})}^{\prime}(1) + \alpha W_{{(\alpha_{2}, \beta_{2})},({\gamma_{2}, \delta_{2}})}(1)).
    \end{align*}
    As a result, by the given conditions,  the desired outcome has been obtained.
\end{proof}

\begin{corollary}\label{cor1}
Setting  $\beta_{1}, \beta_{2}, \delta_{1}=\delta_{2}=1\geq 0$, $\alpha_{1},\alpha_{2}, \gamma_{1}= \gamma_{2}=1> 0$ and $|\sigma|<1$ in above theorem. Also assume $f= h + \bar g \in \hat{H}$ if the inequality $|A_{n}|\leq 1, |B_{n}|\leq 1$ and \\$$ (W_{{(\alpha_{1}, \beta_{1})},(1, 1)}^{\prime}(1)-1) -\alpha(W_{{(\alpha_{1}, \beta_{1})},(1, 1)}(1)-1)+|\sigma|(W_{{(\alpha_{2}, \beta_{2})},(1,1)}^{\prime}(1) + \alpha W_{{(\alpha_{2}, \beta_{2})},(1,1)}(1)\leq (1-\alpha)$$ are satisfied then $L(f)$ given by \eqref{eqwright}is harmonic starlike function, i.e $L(f)\in HS^{*}(\alpha)$.
\end{corollary}

\begin{theorem}
Suppose that $\beta_{1},\beta_{2}, \delta_{1}, \delta_{2}\geq 0$, $\alpha_{1},\alpha_{2}, \gamma_{1}, \gamma_{2}> 0$ and $|\sigma|<1$. Also assume $f= h + \bar g \in SRH(\alpha)$ if the inequality
\begin{equation}\label{eq1}
W_{{(\alpha_{1}, \beta_{1})},({\gamma_{1}, \delta_{1}})}(1)+|\sigma|W_{{(\alpha_{2}, \beta_{2})},({\gamma_{2}, \delta_{2}})}(1) \leq 2
\end{equation}
satisfied, then   $L(f) \in SRH(\alpha)$.   
\end{theorem}
\begin{proof}
  Let us consider  $f = h + \bar g\in SRH(\alpha)$ and in order by using lemma \ref{lem6}, $$\sum_{n=2}^{\infty}(n-\alpha)|A_{n}|+\sum_{n=1}^{\infty}(n+\alpha)|B_{n}|\leq 1-\alpha.$$
  Clearly, we get
  $$|A_{n}|\leq \frac{1-\alpha}{n(n-\alpha)} \quad \forall n \geq 2 \quad and \quad |B_{n}|\leq \frac{1-\alpha}{n(n+\alpha)} \quad \forall n \geq 1.$$
  We have, $$L(f)=z+\sum_{n=2}^{\infty}\frac{\Gamma(\alpha_{1})\Gamma(\gamma_{1})A_{n}z^{n}}{\Gamma(\alpha_{1}  +(n-1)\beta_{1})\Gamma(\gamma_{1}  +(n-1) \delta_{1})}+\sigma\sum_{n=1}^{\infty}\frac{\Gamma(\alpha_{2})\Gamma(\gamma_{2})B_{n}z^{n}}{\Gamma(\alpha_{2}  +(n-1)\beta_{2})\Gamma(\gamma_{2}  +(n-1) \delta_{2})}.$$
It can be noted that by lemma \ref{lem6}, it is enough to prove that 
\begin{align*}
   & P_{1}= \sum_{n=2}^{\infty}(n-\alpha)\left|\frac{\Gamma(\alpha_{1})\Gamma(\gamma_{1})}{\Gamma(\alpha_{1}  +(n-1)\beta_{1})\Gamma(\gamma_{1}  +(n-1) \delta_{1})}\right||A_{n}|\\
&+|\sigma|\sum_{n=1}^{\infty}(n+\alpha)\left|\frac{\Gamma(\alpha_{2})\Gamma(\gamma_{2})}{\Gamma(\alpha_{2}  +(n-1)\beta_{2})\Gamma(\gamma_{2}  +(n-1) \delta_{2})}\right||B_{n}|\leq 1-\alpha.
\end{align*}
We consider,
\begin{align*}
        & P_{1}\leq  \sum_{n=2}^{\infty}\left|\frac{\Gamma(\alpha_{1})\Gamma(\gamma_{1})}{\Gamma(\alpha_{1}  +(n-1)\beta_{1})\Gamma(\gamma_{1}  +(n-1) \delta_{1})}\right|+|\sigma|\sum_{n=1}^{\infty}\left|\frac{\Gamma(\alpha_{2})\Gamma(\gamma_{2})}{\Gamma(\alpha_{2}  +(n-1)\beta_{2})\Gamma(\gamma_{2}  +(n-1) \delta_{2})}\right|\\
        & = \sum_{n=1}^{\infty}\left|\frac{\Gamma(\alpha_{1})\Gamma(\gamma_{1})}{\Gamma(\alpha_{1}  +n\beta_{1})\Gamma(\gamma_{1}  +n \delta_{1})}\right|+|\sigma|\sum_{n=0}^{\infty}\left|\frac{\Gamma(\alpha_{2})\Gamma(\gamma_{2})}{\Gamma(\alpha_{2}  +n\beta_{2})\Gamma(\gamma_{2}  +n \delta_{2})}\right|\\
        &=W_{{(\alpha_{1}, \beta_{1})},({\gamma_{1}, \delta_{1}})}(1)+|\sigma|W_{{(\alpha_{2}, \beta_{2})},({\gamma_{2}, \delta_{2}})}(1).   \end{align*}
 Using the asserted condition of theorem, it is readily to follow the desired result. 
 \end{proof}
In the next theorem determine sufficient condition which assure that a harmonic convex function map unit disk to $D$ into harmonic starlike region.
\begin{theorem}
Suppose that $\beta_{1},\beta_{2}, \delta_{1}, \delta_{2}\geq 0$, $\alpha_{1},\alpha_{2}, \gamma_{1}, \gamma_{2}> 0$ and $|\sigma|<1$. Also assume $f= h + \bar g \in  K_{\hat{H}^{o}}$ if the inequality
\begin{align*}
&W^{\prime\prime}_{{(\alpha_{1}, \beta_{1})},({\gamma_{1}, \delta_{1}})}(1)+(2-\alpha)(W^{\prime}_{{(\alpha_{1}, \beta_{1})},({\gamma_{1}, \delta_{1}})}(1)-1)-\alpha(W^{\prime}_{{(\alpha_{1}, \beta_{1})},({\gamma_{1}, \delta_{1}})}(1)-1)\\
        &+|\sigma|W^{\prime\prime}_{{(\alpha_{2}, \beta_{2})},({\gamma_{2}, \delta_{2}})}(1) +\alpha(W^{\prime}_{{(\alpha_{2}, \beta_{2})},({\gamma_{2}, \delta_{2}})}(1)-W_{{(\alpha_{2}, \beta_{2})},({\gamma_{2}, \delta_{2}})}(1))\leq 2(1-\alpha),
\end{align*}
then   $L(f) \in H + \bar G \in HS^{*}(\alpha).$
\end{theorem}

    \begin{proof}
  Let $f= h + \bar g \in  K_{\hat{H}^{o}}$, where $h$ and $g$ given by the equation \eqref{eqharmonic}  and in view of lemma \ref{lem4},
  it is clear that 
  $$|A_{n}|\leq \frac{(n+1)}{2}, \quad |B_{n}|\leq \frac{n-1}{2}, \quad n=1,2,3,\cdots.$$\\
  Now, we have $$L(f)=z+\sum_{n=2}^{\infty}\frac{\Gamma(\alpha_{1})\Gamma(\gamma_{1})A_{n}z^{n}}{\Gamma(\alpha_{1}  +(n-1)\beta_{1})\Gamma(\gamma_{1}  +(n-1) \delta_{1})}+\sigma\sum_{n=1}^{\infty}\frac{\Gamma(\alpha_{2})\Gamma(\gamma_{2})B_{n}z^{n}}{\Gamma(\alpha_{2}  +(n-1)\beta_{2})\Gamma(\gamma_{2}  +(n-1) \delta_{2})}.$$
It can be noted that by lemma \ref{lem1}, sufficient to prove that 
\begin{align*}
   & P_{1}= \sum_{n=2}^{\infty}(n-\alpha)\left|\frac{\Gamma(\alpha_{1})\Gamma(\gamma_{1})}{\Gamma(\alpha_{1}  +(n-1)\beta_{1})\Gamma(\gamma_{1}  +(n-1) \delta_{1})}\right||A_{n}|\\
        &+|\sigma|\sum_{n=1}^{\infty}(n+\alpha)\left|\frac{\Gamma(\alpha_{2})\Gamma(\gamma_{2})}{\Gamma(\alpha_{2}  +(n-1)\beta_{2})\Gamma(\gamma_{2}  +(n-1) \delta_{2})}\right||B_{n}|\leq 1-\alpha.
        \end{align*}
        Now, we get
        \begin{align*}
        & P_{1}\leq  \frac{1}{2}\left(\sum_{n=2}^{\infty}\left|\frac{(n-\alpha)(n+1)\Gamma(\alpha_{1})\Gamma(\gamma_{1})}{\Gamma(\alpha_{1}  +(n-1)\beta_{1})\Gamma(\gamma_{1}  +(n-1) \delta_{1})}\right|+|\sigma|\sum_{n=1}^{\infty}\left|\frac{(n+\alpha)(n-1)(\Gamma(\alpha_{2})\Gamma(\gamma_{2})}{\Gamma(\alpha_{2}  +(n-1)\beta_{2})\Gamma(\gamma_{2}  +(n-1) \delta_{2})}\right|\right)\\
        & =\frac{1}{2}\left(\sum_{n=1}^{\infty}\left|\frac{(n+1-\alpha)(n+2)\Gamma(\alpha_{1})\Gamma(\gamma_{1})}{\Gamma(\alpha_{1}  +n\beta_{1})\Gamma(\gamma_{1}  +n \delta_{1})}\right|+|\sigma|\sum_{n=0}^{\infty}\left|\frac{(n+1+\alpha)n(\Gamma(\alpha_{2})\Gamma(\gamma_{2})}{\Gamma(\alpha_{2}  +n\beta_{2})\Gamma(\gamma_{2}  +n \delta_{2})}\right|\right)\\
        & =\frac{1}{2}\left(\sum_{n=1}^{\infty}\left|\frac{(n(n+1)+(n+1)(2-\alpha)-\alpha)\Gamma(\alpha_{1})\Gamma(\gamma_{1})}{\Gamma(\alpha_{1}  +n\beta_{1})\Gamma(\gamma_{1}  +n \delta_{1})}\right|+|\sigma|\sum_{n=0}^{\infty}\left|\frac{(n(n+1)+\alpha(n+1)-\alpha)(\Gamma(\alpha_{2})\Gamma(\gamma_{2})}{\Gamma(\alpha_{2}  +n\beta_{2})\Gamma(\gamma_{2}  +n \delta_{2})}\right|\right)\\
        &=\frac{1}{2}\left(W^{\prime\prime}_{{(\alpha_{1}, \beta_{1})},({\gamma_{1}, \delta_{1}})}(1)+(2-\alpha)(W^{\prime}_{{(\alpha_{1}, \beta_{1})},({\gamma_{1}, \delta_{1}})}(1)-1)-\alpha(W^{\prime}_{{(\alpha_{1}, \beta_{1})},({\gamma_{1}, \delta_{1}})}(1)-1)\right)\\
        &+\frac{1}{2}\left(|\sigma|W^{\prime\prime}_{{(\alpha_{2}, \beta_{2})},({\gamma_{2}, \delta_{2}})}(1) +\alpha(W^{\prime}_{{(\alpha_{2}, \beta_{2})},({\gamma_{2}, \delta_{2}})}(1)-W_{{(\alpha_{2}, \beta_{2})},({\gamma_{2}, \delta_{2}})}(1))\right),
        \end{align*}
  \end{proof}
  follows by given conditions, the required outcome has been established.
\section{Harmonic Convexity}\label{sec4}
In this section, we discuss some sufficient conditions in such a way that the harmonic mapping involving four-parameter Wright function satisfies harmonic convexity.
\begin{theorem}
    Assume that $\beta_{1},\beta_{2}, \delta_{1}, \delta_{2}\geq 0$, $\alpha_{1},\alpha_{2}, \gamma_{1}, \gamma_{2}> 0$ and $|\sigma|<1$. Also given $f= h + \bar g \in \hat{H}$ if the inequality $|A_{n}|\leq 1, |B_{n}|\leq 1$ and \\$$W_{{(\alpha_{1}, \beta_{1})},({\gamma_{1}, \delta_{1}})}^{\prime\prime}(1) +(1-\alpha)(W_{{(\alpha_{1}, \beta_{1})},({\gamma_{1}, \delta_{1}})}^{\prime}(1))+|\sigma|(W_{{(\alpha_{2}, \beta_{2})},({\gamma_{2}, \delta_2})}^{\prime\prime}(1) + (1+\alpha) W_{{(\alpha_{2}, \beta_{2})},({\gamma_{2}, \delta_{2}})}^{\prime}(1))\leq \alpha$$ are satisfied then $L(f)$ is harmonic convex function, i.e $L(f)\in HK(\alpha)$.
\end{theorem}
\begin{proof}
  Let us consider $f= h + \bar g \in  K_{\hat{H}^{o}}$, we have to show $L(f)\in HK(\alpha)$,\\
Where,
$$L(f)=z+\sum_{n=2}^{\infty}\frac{\Gamma(\alpha_{1})\Gamma(\gamma_{1})A_{n}z^{n}}{\Gamma(\alpha_{1}  +(n-1)\beta_{1})\Gamma(\gamma_{1}  +(n-1) \delta_{1})}+\sigma\sum_{n=1}^{\infty}\frac{\Gamma(\alpha_{2})\Gamma(\gamma_{2})B_{n}z^{n}}{\Gamma(\alpha_{2}  +(n-1)\beta_{2})\Gamma(\gamma_{2}  +(n-1) \delta_{2})}.$$
 It can be followed by lemma \ref{lem2}, enough to prove that
 \begin{align*}
        S_{3}=&\sum_{n=2}^{\infty}n(n-\alpha)\left|\frac{\Gamma(\alpha_{1})\Gamma(\gamma_{1})}{\Gamma(\alpha_{1}  +(n-1)\beta_{1})\Gamma(\gamma_{1}  +(n-1) \delta_{1})}\right||A_{n}\\
&+|\sigma|\sum_{n=1}^{\infty}n(n+\alpha)\left|\frac{\Gamma(\alpha_{2})\Gamma(\gamma_{2})}{\Gamma(\alpha_{2}  +(n-1)\beta_{2})\Gamma(\gamma_{2}  +(n-1) \delta_{2})}\right||B_{n}|\leq 1-\alpha.
\end{align*}
We have
\begin{align*}
         S_{3}\leq &\sum_{n=2}^{\infty}n(n-\alpha)\left|\frac{\Gamma(\alpha_{1})\Gamma(\gamma_{1})}{\Gamma(\alpha_{1}  +(n-1)\beta_{1})\Gamma(\gamma_{1}  +(n-1) \delta_{1})}\right|\\
&+|\sigma|\sum_{n=1}^{\infty}n(n+\alpha)\left|\frac{\Gamma(\alpha_{2})\Gamma(\gamma_{2})}{\Gamma(\alpha_{2}  +(n-1)\beta_{2})\Gamma(\gamma_{2}  +(n-1) \delta_{2})}\right|\\
         &  = \sum_{n=1}^{\infty}(n+1)(n+1-\alpha)\left|\frac{\Gamma(\alpha_{1})\Gamma(\gamma_{1})}{\Gamma(\alpha_{1}  +n \beta_{1})\Gamma(\gamma_{1}  +n \delta_{1})}\right|\\
        & + |\sigma|\sum_{n=0}^{\infty}(n+1)(n+1+\alpha)\left|\frac{\Gamma(\alpha_{2})\Gamma(\gamma_{2})}{\Gamma(\alpha_{2}  +(n-1)\beta_{2})\Gamma(\gamma_{2}  +(n-1) \delta_{2})}\right|\\
        &  = \sum_{n=1}^{\infty}(n+1)n\left|\frac{\Gamma(\alpha_{1})\Gamma(\gamma_{1})}{\Gamma(\alpha_{1}  +n \beta_{1})\Gamma(\gamma_{1}  +n \delta_{1})}\right|+\sum_{n=1}^{\infty}(n+1)(1-\alpha)\left|\frac{\Gamma(\alpha_{1})\Gamma(\gamma_{1})}{\Gamma(\alpha_{1}  +n \beta_{1})\Gamma(\gamma_{1}  +n \delta_{1})}\right|\\
        & + |\sigma|\left(\sum_{n=0}^{\infty}\left|\frac{(n+1)n\Gamma(\alpha_{2})\Gamma(\gamma_{2})}{\Gamma(\alpha_{2}  +(n-1)\beta_{2})\Gamma(\gamma_{2}  +(n-1) \delta_{2})}\right|+|\sum_{n=0}^{\infty}\left|\frac{(n+1)(1+\alpha)\Gamma(\alpha_{2})\Gamma(\gamma_{2})}{\Gamma(\alpha_{2}  +(n-1)\beta_{2})\Gamma(\gamma_{2}  +(n-1) \delta_{2})}\right|\right)\\
        & = W_{{(\alpha_{1}, \beta_{1})},({\gamma_{1}, \delta_{1}})}^{\prime\prime}(1) +(1-\alpha)(W_{{(\alpha_{1}, \beta_{1})},({\gamma_{1}, \delta_{1}})}^{\prime}(1))+|\sigma|(W_{{(\alpha_{2}, \beta_{2})},({\gamma_{2}, \delta_2})}^{\prime\prime}(1) + (1+\alpha) W_{{(\alpha_{2}, \beta_{2})},({\gamma_{2}, \delta_{2}})}^{\prime}(1)).
    \end{align*}
 As a result, by the given hypothesis,  the required theorem has been developed.
\end{proof}

\begin{remark}\label{reamrk1}
  Analogous of above theorem, if $\beta_{1},\beta_{2}, \delta_{1}=\delta_{2}=1\geq 0$, $\alpha_{1},\alpha_{2}, \gamma_{1}=\gamma_{2}=1> 0$ and $|\sigma|<1$. Also let $f= h + \bar g \in \hat{H}$ if the inequality $|A_{n}|\leq 1, |B_{n}|\leq 1$ and \\$$W_{{(\alpha_{1}, \beta_{1})},(1,1)}^{\prime\prime}(1) +(1-\alpha)(W_{{(\alpha_{1}, \beta_{1})},(1,1)}^{\prime}(1))+|\sigma|(W_{{(\alpha_{2}, \beta_{2})},(1,1)}^{\prime\prime}(1) + (1+\alpha) W_{{(\alpha_{2}, \beta_{2})},(1,1)}^{\prime}(1))\leq \alpha$$ are satisfied then $L(f)$ given by \eqref{eqwright} map harmonic mapping to harmonic convex mapping under unit disk, i.e $L(f)\in HK(\alpha)$.
\end{remark}

In the next theorem, we investigate sufficient condition which assure that a harmonic convex function map unit disk to $D$ into harmonic convex domain.
\begin{theorem}
Let us consider that $\beta_{1},\beta_{2}, \delta_{1}, \delta_{2}\geq 0$, $\alpha_{1},\alpha_{2}, \gamma_{1}, \gamma_{2}> 0$ and $|\sigma|<1$. Also provided that $f= h + \bar g \in  K_{\hat{H}^{o}}$ if the inequality holds
\begin{align*}
&\left(W^{\prime\prime\prime}_{{(\alpha_{1}, \beta_{1})},({\gamma_{1}, \delta_{1}})}(1)+(4-\alpha)W^{\prime\prime}_{{(\alpha_{1}, \beta_{1})},({\gamma_{1}, \delta_{1}})}(1)+2(1-\alpha)(W^{\prime}_{{(\alpha_{1}, \beta_{1})},({\gamma_{1}, \delta_{1}})}(1)-1)\right)\\
        &+|\sigma|\left(W^{\prime\prime\prime}_{{(\alpha_{2}, \beta_{2})},({\gamma_{2}, \delta_{2}})}(1)+(2+\alpha)W^{\prime\prime}_{{(\alpha_{2}, \beta_{2})},({\gamma_{2}, \delta_{2}})}(1)\right)\leq 2(1-\alpha),
\end{align*}
then   $L(f) \in H + \bar G \in HK(\alpha).$
\end{theorem}

    \begin{proof}
   Suppose that $f= h + \bar g \in  K_{\hat{H}^{o}}$, where $h$ and $g$ given by the equation \eqref{eqharmonic}  and follow by lemma \ref{lem4},
  it is clear that 
  $$|A_{n}|\leq \frac{(n+1)}{2}, \quad |B_{n}|\leq \frac{n-1}{2}, \quad n=1,2,3,\cdots.$$\\
  We have $$L(f)=z+\sum_{n=2}^{\infty}\frac{\Gamma(\alpha_{1})\Gamma(\gamma_{1})A_{n}z^{n}}{\Gamma(\alpha_{1}  +(n-1)\beta_{1})\Gamma(\gamma_{1}  +(n-1) \delta_{1})}+\sigma\sum_{n=1}^{\infty}\frac{\Gamma(\alpha_{2})\Gamma(\gamma_{2})B_{n}z^{n}}{\Gamma(\alpha_{2}  +(n-1)\beta_{2})\Gamma(\gamma_{2}  +(n-1) \delta_{2})}.$$
With the help of lemma \ref{lem2}, It is adequate to prove that 
\begin{align*}
   & P_{1}= \sum_{n=2}^{\infty}n(n-\alpha)\left|\frac{\Gamma(\alpha_{1})\Gamma(\gamma_{1})}{\Gamma(\alpha_{1}  +(n-1)\beta_{1})\Gamma(\gamma_{1}  +(n-1) \delta_{1})}\right||A_{n}|\\
&+|\sigma|\sum_{n=1}^{\infty}n(n+\alpha)\left|\frac{\Gamma(\alpha_{2})\Gamma(\gamma_{2})}{\Gamma(\alpha_{2}  +(n-1)\beta_{2})\Gamma(\gamma_{2}  +(n-1) \delta_{2})}\right||B_{n}|\leq 1-\alpha.
\end{align*}
We obtain,
\begin{align*}
        & P_{1}\leq  \frac{1}{2}\left(\sum_{n=2}^{\infty}\left|\frac{n(n-\alpha)(n+1)\Gamma(\alpha_{1})\Gamma(\gamma_{1})}{\Gamma(\alpha_{1}  +(n-1)\beta_{1})\Gamma(\gamma_{1}  +(n-1) \delta_{1})}\right|+|\sigma|\sum_{n=1}^{\infty}\left|\frac{n(n+\alpha)(n-1)(\Gamma(\alpha_{2})\Gamma(\gamma_{2})}{\Gamma(\alpha_{2}  +(n-1)\beta_{2})\Gamma(\gamma_{2}  +(n-1) \delta_{2})}\right|\right)\\
        & =\frac{1}{2}\left(\sum_{n=1}^{\infty}\left|\frac{(n+1)(n+1-\alpha)(n+2)\Gamma(\alpha_{1})\Gamma(\gamma_{1})}{\Gamma(\alpha_{1}  +n\beta_{1})\Gamma(\gamma_{1}  +n \delta_{1})}\right|+|\sigma|\sum_{n=0}^{\infty}\left|\frac{(n+1)(n+1+\alpha)n(\Gamma(\alpha_{2})\Gamma(\gamma_{2})}{\Gamma(\alpha_{2}  +n\beta_{2})\Gamma(\gamma_{2}  +n \delta_{2})}\right|\right)\\
        &=\frac{1}{2}\left(\sum_{n=1}^{\infty}\left|\frac{(n(n+1)(n-1)+(4-\alpha)n(n+1)+2(n+1)(1-\alpha))\Gamma(\alpha_{1})\Gamma(\gamma_{1})}{\Gamma(\alpha_{1}  +n\beta_{1})\Gamma(\gamma_{1}  +n \delta_{1})}\right|\right)\\
        &+\frac{|\sigma|}{2}\left(\sum_{n=0}^{\infty}\left|\frac{((n+1)n(n-1)+n(n+1)(2+\alpha))(\Gamma(\alpha_{2})\Gamma(\gamma_{2})}{\Gamma(\alpha_{2}  +n\beta_{2})\Gamma(\gamma_{2}  +n \delta_{2})}\right|\right)\\
        &=\frac{1}{2}\left(W^{\prime\prime\prime}_{{(\alpha_{1}, \beta_{1})},({\gamma_{1}, \delta_{1}})}(1)+(4-\alpha)W^{\prime\prime}_{{(\alpha_{1}, \beta_{1})},({\gamma_{1}, \delta_{1}})}(1)+2(1-\alpha)(W^{\prime}_{{(\alpha_{1}, \beta_{1})},({\gamma_{1}, \delta_{1}})}(1)-1)\right)\\
        &+\frac{|\sigma|}{2}\left(W^{\prime\prime\prime}_{{(\alpha_{2}, \beta_{2})},({\gamma_{2}, \delta_{2}})}(1)+(2+\alpha)W^{\prime\prime}_{{(\alpha_{2}, \beta_{2})},({\gamma_{2}, \delta_{2}})}(1)\right). 
        \end{align*}
        The asserted condition of theorem obtained, readily to follow the desired result. 
 \end{proof}

  \begin{theorem}
  Assume that $\beta_{1},\beta_{2}, \delta_{1}, \delta_{2}\geq 0$, $\alpha_{1},\alpha_{2}, \gamma_{1}, \gamma_{2}> 0$ and $|\sigma|<1$. Also let us consider $f= h + \bar g \in KRH(\alpha)$ if the inequality $$W_{{(\alpha_{1}, \beta_{1})},({\gamma_{1}, \delta_{1}})}(1)+|\sigma|W_{{(\alpha_{2}, \beta_{2})},({\gamma_{2}, \delta_{2}})}(1) \leq 2-\alpha$$ satisfied, then   $L(f) \in KRH(\alpha).$
\end{theorem}

\begin{proof}
  Given hypothesis is  $f = h + \bar g\in KRH(\alpha)$ and in order to use of lemma \ref{lem6}  provide $$\sum_{n=2}^{\infty}(n(n-\alpha))|A_{n}|+\sum_{n=1}^{\infty}(n(n+\alpha))|A_{n}|\leq 1-\alpha.$$
  Clearly, we get
  $$|A_{n}|\leq \frac{1-\alpha}{n(n-\alpha)} \quad \forall n \geq 2 \quad and \quad |B_{n}|\leq \frac{1-\alpha}{n(n+\alpha)} \quad \forall n \geq 1.$$
  Now, we have $$L(f)=z+\sum_{n=2}^{\infty}\frac{\Gamma(\alpha_{1})\Gamma(\gamma_{1})A_{n}z^{n}}{\Gamma(\alpha_{1}  +(n-1)\beta_{1})\Gamma(\gamma_{1}  +(n-1) \delta_{1})}+\sigma\sum_{n=1}^{\infty}\frac{\Gamma(\alpha_{2})\Gamma(\gamma_{2})B_{n}z^{n}}{\Gamma(\alpha_{2}  +(n-1)\beta_{2})\Gamma(\gamma_{2}  +(n-1) \delta_{2})}.$$
It can be observe that by lemma \ref{lem6}, adequate to prove that 
\begin{align*}
   & P_{1}= \sum_{n=2}^{\infty}n(n-\alpha)\left|\frac{\Gamma(\alpha_{1})\Gamma(\gamma_{1})}{\Gamma(\alpha_{1}  +(n-1)\beta_{1})\Gamma(\gamma_{1}  +(n-1) \delta_{1})}\right||A_{n}|\\
&+|\sigma|\sum_{n=1}^{\infty}n(n+\alpha)\left|\frac{\Gamma(\alpha_{2})\Gamma(\gamma_{2})}{\Gamma(\alpha_{2}  +(n-1)\beta_{2})\Gamma(\gamma_{2}  +(n-1) \delta_{2})}\right||B_{n}|\leq 1-\alpha.
\end{align*}
We consider,
\begin{align*}
        & P_{1}\leq  \sum_{n=2}^{\infty}\left|\frac{\Gamma(\alpha_{1})\Gamma(\gamma_{1})}{\Gamma(\alpha_{1}  +(n-1)\beta_{1})\Gamma(\gamma_{1}  +(n-1) \delta_{1})}\right|+|\sigma|\sum_{n=1}^{\infty}\left|\frac{\Gamma(\alpha_{2})\Gamma(\gamma_{2})}{\Gamma(\alpha_{2}  +(n-1)\beta_{2})\Gamma(\gamma_{2}  +(n-1) \delta_{2})}\right|\\
        & = \sum_{n=1}^{\infty}\left|\frac{\Gamma(\alpha_{1})\Gamma(\gamma_{1})}{\Gamma(\alpha_{1}  +n\beta_{1})\Gamma(\gamma_{1}  +n \delta_{1})}\right|+|\sigma|\sum_{n=0}^{\infty}\left|\frac{\Gamma(\alpha_{2})\Gamma(\gamma_{2})}{\Gamma(\alpha_{2}  +n\beta_{2})\Gamma(\gamma_{2}  +n \delta_{2})}\right|\\
        &=W_{{(\alpha_{1}, \beta_{1})},({\gamma_{1}, \delta_{1}})}(1)+|\sigma|W_{{(\alpha_{2}, \beta_{2})},({\gamma_{2}, \delta_{2}})}(1).   \end{align*}
 Using the stated condition of theorem established, readily to follow the required result. 
 \end{proof}

In the below section, we obtain some sufficient conditions in such a way that the harmonic mapping involving four-parameter Wright function satisfies harmonic close-to-convex domain.

In this we determine adequate condition which guarantee that a harmonic mapping map unit disk to $D$ into harmonic close-to-convex domain.
\section{Harmonic Close-to-convexity}\label{sec5}
\begin{theorem}
Suppose that $\beta_{1},\beta_{2}, \delta_{1}, \delta_{2}\geq 0$, $\alpha_{1},\alpha_{2}, \gamma_{1}, \gamma_{2}> 0$ and $|\sigma|<1$. Also let us assume $f= h + \bar g \in \hat{H}$ if the following conditions
\begin{enumerate}\label{eq1}

\rm(i) $\sum_{n=2}^{\infty}n|A_{n}|+\sum_{n=1}^{\infty}n|B_{n}|\leq1, \quad |B_{1}|<1$,\\

\rm(ii)$[W_{{(\alpha_{1}, \beta_{1})},({\gamma_{1}, \delta_{1}})}(1)+W_{{(\alpha_{2}, \beta_{2})},({\gamma_{2}, \delta_{2}})}(1)-2] \leq (1-|B_{1}|)$
\end{enumerate}
satisfied, then   $L(f) \in H + \bar G \in C_{\hat{H}}.$  \end{theorem}
\begin{proof}
    Let $f =h + \bar g$ and $L(f)=H +\bar G, B_{1} \neq 0,$. $h,g$ and $H,G$ given by equation \eqref{eqharmonic} and \eqref{eqHarmonic} respectively.  We have to prove that $L(f) = H + \bar{G} \in C_{\hat{H}}$.   Since $H$ and $G$ are analytic function, it can be noted that $H^{\prime}(0)=1 >|B_{1}|=|G^{\prime}(0)|.$\\ 
    Suppose that $\frac{H+\epsilon G}{1+\epsilon B_{1}} =z + \sum_{n=2}^{\infty} nt_{n},$
where
$$t_{n}=\frac{\Gamma(\alpha_{1})\Gamma(\gamma_{1})A_{n}}{\Gamma(\alpha_{1}  +(n-1)\beta_{1})\Gamma(\gamma_{1}  +(n-1) \delta_{1})(1+\epsilon B_{1})}+\epsilon\frac{ \Gamma(\alpha_{2})\Gamma(\gamma_{2})B_{n}}{\Gamma(\alpha_{2}  +(n-1)\beta_{2})\Gamma(\gamma_{2}  +(n-1) \delta_{2})(1+\epsilon B_{1})},$$ for all $|\epsilon|=1.$ By taking lemma \ref{lem5} into the account, it is sufficient to show that $\sum_{n=2}^{\infty} nt_{n}\leq 1$ to get the required result.
Using condition $(i), |A_{n}|\leq \frac{1}{n}, |B_{n}|\leq \frac{1}{n}$  for all $n\geq 2$, we have
\begin{align*}
&\sum_{n=2}^{\infty} n t_{n}\leq\\
&  \left[\sum_{n=2}^{\infty}\frac{\Gamma(\alpha_{1})\Gamma(\gamma_{1})}{\Gamma(\alpha_{1}  +(n-1)\beta_{1})\Gamma(\gamma_{1}  +(n-1) \delta_{1})(1-| B_{1}|)}+\epsilon\frac{ \Gamma(\alpha_{2})\Gamma(\gamma_{2})}{\Gamma(\alpha_{2}  +(n-1)\beta_{2})\Gamma(\gamma_{2}  +(n-1) \delta_{2})(1- |B_{1}|)}\right]\\
&\left[\sum_{n=1}^{\infty}\frac{\Gamma(\alpha_{1})\Gamma(\gamma_{1})}{\Gamma(\alpha_{1}  +n\beta_{1})\Gamma(\gamma_{1}  +n\delta_{1})(1-| B_{1}|)}+\epsilon\frac{ \Gamma(\alpha_{2})\Gamma(\gamma_{2})}{\Gamma(\alpha_{2}  +n\beta_{2})\Gamma(\gamma_{2}  +n \delta_{2})(1- |B_{1}|)}\right]\\
&=\frac{1}{(1-|B_{1}|)}[W_{{(\alpha_{1}, \beta_{1})},({\gamma_{1}, \delta_{1}})}(1)+W_{{(\alpha_{2}, \beta_{2})},({\gamma_{2}, \delta_{2}})}(1)-2]. 
\end{align*}
Follow by lemma \ref{lem3}, we obtain the desired outcome that $L(f)\in  C_{\hat{H}}. $
\end{proof}

In the next we investigate adequate condition which assure that a harmonic convex function map unit disk to $D$ into harmonic close-to-convex region.

\begin{theorem}
Let us consider $\beta_{1},\beta_{2}, \delta_{1}, \delta_{2}\geq 0$, $\alpha_{1},\alpha_{2}, \gamma_{1}, \gamma_{2}> 0$ and $|\sigma|<1$. Also suppose $f= h + \bar g \in  K_{\hat{H}^{o}}$ if the inequality
\begin{equation}\label{eq2}
[W_{{(\alpha_{1}, \beta_{1})},({\gamma_{1}, \delta_{1}})}^{\prime\prime}(1)+2W_{{(\alpha_{1}, \beta_{1})},({\gamma_{1}, \delta_{1}})}^{\prime}(1)+ W_{{(\alpha_{2}, \beta_{2})},({\gamma_{2}, \delta_{2}})}^{\prime\prime}(1)]\leq 4
\end{equation}
satisfied, then   $L(f) \in H + \bar G \in C_{\hat{H}}.$    \end{theorem}
\begin{proof}
By given assumption $f =h + \bar g \in K_{\hat{H}^{o}}$ and $L(f)=H +\bar G, B_{1} = 0,$ where $h,g$ and $H,G$ given by equation \eqref{eqharmonic} and \eqref{eqHarmonic} respectively.  Since $H$ and $G$ are analytic function, it can be noted that $H^{\prime}(0)=1 >|B_{1}|=|G^{\prime}(0)|.$ 
    In view of lemma \ref{lem3}, it suffices to prove that  $H+\epsilon G =z + \sum_{n=2}^{\infty} nt_{n} \in C_{\hat{H}},$
where
$$t_{n}=\frac{\Gamma(\alpha_{1})\Gamma(\gamma_{1})A_{n}}{\Gamma(\alpha_{1}  +(n-1)\beta_{1})\Gamma(\gamma_{1}  +(n-1) \delta_{1})(1+\epsilon B_{1})}+\epsilon\frac{ \Gamma(\alpha_{2})\Gamma(\gamma_{2})B_{n}}{\Gamma(\alpha_{2}  +(n-1)\beta_{2})\Gamma(\gamma_{2}  +(n-1) \delta_{2})(1+\epsilon B_{1})},$$ for all $|\epsilon|=1.$ With the help of  lemma \ref{lem5}, it is adequate to prove that $\sum_{n=2}^{\infty} nt_{n}\leq 1$ to get the required outcome.
Employing lemma \ref{lem4}, \begin{align*}
    |A_{n}|\leq \frac{n+1}{2}, |B_{n}|\leq \frac{n-1}{2} \quad \forall \quad n\geq 2,
    \end{align*}we have
\begin{align*}
&\sum_{n=2}^{\infty} n t_{n}\leq\\
&  \frac{1}{2}\left[\sum_{n=1}^{\infty}\frac{(n+1)(n+2)\Gamma(\alpha_{1})\Gamma(\gamma_{1})}{\Gamma(\alpha_{1}  +n\beta_{1})\Gamma(\gamma_{1}  +n \delta_{1})}+\epsilon\frac{n(n+1) \Gamma(\alpha_{2})\Gamma(\gamma_{2})}{\Gamma(\alpha_{2}  +n\beta_{2})\Gamma(\gamma_{2}  +n \delta_{2})}\right]\\
&\leq \frac{1}{2}\left[\sum_{n=1}^{\infty}\frac{n(n+1)\Gamma(\alpha_{1})\Gamma(\gamma_{1})}{\Gamma(\alpha_{1}  +n\beta_{1})\Gamma(\gamma_{1}  +n \delta_{1})}+\frac{2(n+1)\Gamma(\alpha_{1})\Gamma(\gamma_{1})}{\Gamma(\alpha_{1}  +n\beta_{1})\Gamma(\gamma_{1}  +n \delta_{1})}+\epsilon\frac{n(n+1) \Gamma(\alpha_{2})\Gamma(\gamma_{2})}{\Gamma(\alpha_{2}  +n\beta_{2})\Gamma(\gamma_{2}  +n \delta_{2})}\right]\\
&=\frac{1}{2}[W_{{(\alpha_{1}, \beta_{1})},({\gamma_{1}, \delta_{1}})}^{\prime\prime}(1)+2W_{{(\alpha_{1}, \beta_{1})},({\gamma_{1}, \delta_{1}})}^{\prime}(1)-2+\epsilon W_{{(\alpha_{2}, \beta_{2})},({\gamma_{2}, \delta_{2}})}^{\prime\prime}(1)]. 
\end{align*}
Follow by asserted condition of theorem, we establish the result that $L(f)\in  C_{\hat{H}}. $
\end{proof}

\begin{theorem}\label{thm9}
Suppose that $\beta_{1},\beta_{2}, \delta_{1}, \delta_{2}\geq 0$, $\alpha_{1},\alpha_{2}, \gamma_{1}, \gamma_{2}> 0$ and $|\sigma|<1$. Also assume $f=h + \bar g \in C_{\hat{H}^{o}}(S_{\hat{H}^{o}} or T_{\hat{H}^{o}})$ if the inequality
\begin{align*}\label{eq2}
&2W_{{(\alpha_{1}, \beta_{1})},({\gamma_{1}, \delta_{1}})}^{\prime\prime\prime}(1)+9W_{{(\alpha_{1}, \beta_{1})},({\gamma_{1}, \delta_{1}})}^{\prime\prime}(1)+6(W_{{(\alpha_{2}, \beta_{2})},({\gamma_{2}, \delta_{2}})}^{\prime}(1)-1)+ W_{{(\alpha_{2}, \beta_{2})},({\gamma_{2}, \delta_{2}})}^{\prime\prime\prime}(1)\\
&+3 W_{{(\alpha_{2}, \beta_{2})},({\gamma_{2}, \delta_{2}})}^{\prime\prime}(1)\leq 6
\end{align*}
satisfied, then   $L(f) = H + \bar G \in C_{\hat{H}^{o}}.$    \end{theorem}

\begin{proof}
    Consider $f=h + \bar g \in C_{\hat{H}^{o}}(S_{\hat{H}^{o}} or T_{\hat{H}^{o}})$, $L(f) = H + \bar G$, where $h,g$ and $H,G$ given by equation \eqref{eqharmonic} and \eqref{eqHarmonic} respectively. We have to show $L(f) \in C_{\hat{H}^{o}}.$ In view of lemma \ref{lem3}, it is enough to prove that  $H+\epsilon G =z + \sum_{n=2}^{\infty} nt_{n} \in C_{\hat{H}},$
where
$$t_{n}=\frac{\Gamma(\alpha_{1})\Gamma(\gamma_{1})A_{n}}{\Gamma(\alpha_{1}  +(n-1)\beta_{1})\Gamma(\gamma_{1}  +(n-1) \delta_{1})}+\epsilon\frac{ \Gamma(\alpha_{2})\Gamma(\gamma_{2})B_{n}}{\Gamma(\alpha_{2}  +(n-1)\beta_{2})\Gamma(\gamma_{2}  +(n-1) \delta_{2})},$$ for all $|\epsilon|=1.$ By taking lemma \ref{lem5} into the account, it is sufficient to show that $\sum_{n=2}^{\infty} nt_{n}\leq 1$ to get the required result.
Using lemma \ref{lem7}, 
\begin{align*}  
|A_{n}|\leq \frac{(2n+1)(n+1)}{6}, |B_{n}|\leq \frac{(2n-1)(n-1)}{6}  \quad \forall \quad n\geq 2,
\end{align*} we have
\begin{align*}
&\sum_{n=2}^{\infty} n t_{n}\leq\\
&\frac{1}{6}\sum_{n=2}^{\infty}  \left[\frac{\Gamma(\alpha_{1})\Gamma(\gamma_{1})n(2n+1)(n+1)}{\Gamma(\alpha_{1}  +(n-1)\beta_{1})\Gamma(\gamma_{1}  +(n-1) \delta_{1})}+\epsilon\frac{ \Gamma(\alpha_{2})\Gamma(\gamma_{2})n(2n-1)(n-1)}{\Gamma(\alpha_{2}  +(n-1)\beta_{2})\Gamma(\gamma_{2}  +(n-1) \delta_{2})}\right]\\
& = \frac{1}{6}\left[\sum_{n=1}^{\infty}\frac{(n+1)(2n+3)(n+2)\Gamma(\alpha_{1})\Gamma(\gamma_{1})}{\Gamma(\alpha_{1}  +n\beta_{1})\Gamma(\gamma_{1}  +n \delta_{1})}+\epsilon\frac{n(n+1)(2n+1) \Gamma(\alpha_{2})\Gamma(\gamma_{2})}{\Gamma(\alpha_{2}  +n\beta_{2})\Gamma(\gamma_{2}  +n \delta_{2})}\right]\\
&\leq \frac{1}{6}\left[\sum_{n=1}^{\infty}\frac{2n(n+1)(n-1)\Gamma(\alpha_{1})\Gamma(\gamma_{1})}{\Gamma(\alpha_{1}  +n\beta_{1})\Gamma(\gamma_{1}  +n \delta_{1})}+\frac{9n(n+1)\Gamma(\alpha_{1})\Gamma(\gamma_{1})}{\Gamma(\alpha_{1}  +n\beta_{1})\Gamma(\gamma_{1}  +n \delta_{1})}+\frac{6(n+1)\Gamma(\alpha_{1})\Gamma(\gamma_{1})}{\Gamma(\alpha_{1}  +n\beta_{1})\Gamma(\gamma_{1}  +n \delta_{1})}\right]\\
&+\frac{1}{6}\left[\left(\sum_{n=1}^{\infty}\frac{2n(n+1)(n-1) \Gamma(\alpha_{2})\Gamma(\gamma_{2})}{\Gamma(\alpha_{2}  +n\beta_{2})\Gamma(\gamma_{2}  +n \delta_{2})}+\frac{3n(n+1) \Gamma(\alpha_{2})\Gamma(\gamma_{2})}{\Gamma(\alpha_{2}  +n\beta_{2})\Gamma(\gamma_{2}  +n \delta_{2})}\right)\right]\\
&=\frac{1}{6}[2W_{{(\alpha_{1}, \beta_{1})},({\gamma_{1}, \delta_{1}})}^{\prime\prime\prime}(1)+9W_{{(\alpha_{1}, \beta_{1})},({\gamma_{1}, \delta_{1}})}^{\prime\prime}(1)+6(W_{{(\alpha_{1}, \beta_{1})},({\gamma_{1}, \delta_{1}})}^{\prime}(1)-1)+2 W_{{(\alpha_{2}, \beta_{2})},({\gamma_{2}, \delta_{2}})}^{\prime\prime\prime}(1)]\\
&+\frac{1}{2} W_{{(\alpha_{2}, \beta_{2})},({\gamma_{2}, \delta_{2}})}^{\prime\prime}(1),
\end{align*}
Taking identities $(n+1)(2n+3)(n+2)=2n(n+1)(n-1)+9n(n+1)+6(n+1)$\\
$n(n+1)(2n+1)=2n(n+1)(n-1)+3n(n+1).$\\
Hence, proof is readily to follow by hypothesis of theorem statement.
\end{proof}

In the next result we investigate sufficient condition which enough that harmonic close-to-convex map the unit disk $D$ to harmonic close-to-convex region.

\begin{theorem}
Assume that $\beta_{1},\beta_{2}, \delta_{1}, \delta_{2}\geq 0$, $\alpha_{1},\alpha_{2}, \gamma_{1}, \gamma_{2}> 0$ and $|\sigma|<1$. Also let $f=h + \bar g \in C_{\hat{H}}$ if the following inequality hold
\begin{align*}
&=2W_{{(\alpha_{1}, \beta_{1})},({\gamma_{1}, \delta_{1}})}^{\prime\prime\prime}(1)+9W_{{(\alpha_{1}, \beta_{1})},({\gamma_{1}, \delta_{1}})}^{\prime\prime}(1)+6(W_{{(\alpha_{2}, \beta_{2})},({\gamma_{2}, \delta_{2}})}^{\prime}(1)-1)\\
&+|B_{1|}|\left(2 W_{{(\alpha_{1}, \beta_{1})},({\gamma_{1}, \delta_{1}})}^{\prime\prime\prime}(1)]+ 3W_{{(\alpha_{1}, \beta_{1})},({\gamma_{1}, \delta_{1}})}^{\prime\prime}(1)\right)\\
&+2 W_{{(\alpha_{2}, \beta_{2})},({\gamma_{2}, \delta_{2}})}^{\prime\prime\prime}(1)+3 W_{{(\alpha_{2}, \beta_{2})},({\gamma_{2}, \delta_{2}})}^{\prime\prime}(1)\\
&+|B_{1}|\left(2W_{{(\alpha_{2}, \beta_{2})},({\gamma_{2}, \delta_{2}})}^{\prime\prime\prime}(1)+9W_{{(\alpha_{2}, \beta_{2})},({\gamma_{2}, \delta_{2}})}^{\prime\prime}(1)+6(W_{{(\alpha_{2}, \beta_{2})},({\gamma_{2}, \delta_{2}})}^{\prime}(1)-1)\right)\leq 6(1-|B_{1}|),
\end{align*}
 then   $L(f) = H + \bar G \in C_{\hat{H}}.$    \end{theorem}

\begin{proof}
    Let $f=h + \bar g \in C_{\hat{H}}$, $L(f) = H + \bar G$, where $h,g$ and $H,G$ given by equation \eqref{eqharmonic} and \eqref{eqHarmonic} respectively. We have to prove $L(f) \in C_{\hat{H}}$, where G and H are analytic function in $D$ and $H^{\prime}(0)=1>|B_{1}|=|G^{\prime}(0)|$. In view of lemma \ref{lem3}, it is enough to prove that  $\frac{H+\epsilon G}{1+\epsilon B_{1}} =z + \sum_{n=2}^{\infty} nt_{n} \in C,$
where
$$t_{n}=\frac{\Gamma(\alpha_{1})\Gamma(\gamma_{1})A_{n}}{\Gamma(\alpha_{1}  +(n-1)\beta_{1})\Gamma(\gamma_{1}  +(n-1) \delta_{1})(1+\epsilon B_{1})}+\epsilon\frac{ \Gamma(\alpha_{2})\Gamma(\gamma_{2})B_{n}}{\Gamma(\alpha_{2}  +(n-1)\beta_{2})\Gamma(\gamma_{2}  +(n-1) \delta_{2})(1+\epsilon B_{1})},$$ for all $|\epsilon|=1.$ By taking lemma \ref{lem5} into the account, it is enough to prove that $\sum_{n=2}^{\infty} nt_{n}\leq 1$ to get the stated outcome.
Using lemma \ref{lemclosetoconvex}, 
we get
\begin{align*}
&\sum_{n=2}^{\infty} n t_{n}\leq\\
&\frac{1}{6(1-|B_{1}|)}\sum_{n=2}^{\infty}  \left[\frac{\Gamma(\alpha_{1})\Gamma(\gamma_{1})n(2n+1)(n+1)}{\Gamma(\alpha_{1}  +(n-1)\beta_{1})\Gamma(\gamma_{1}  +(n-1) \delta_{1})}+\frac{ \Gamma(\alpha_{1})\Gamma(\gamma_{1})n(2n-1)(n-1)|B_{1}|}{\Gamma(\alpha_{1}  +(n-1)\beta_{1})\Gamma(\gamma_{1}  +(n-1) \delta_{1})}\right]\\
&+\frac{1}{6(1-|B_{1}|)}\sum_{n=2}^{\infty}  \left[\frac{\Gamma(\alpha_{2})\Gamma(\gamma_{2})n(2n-1)(n-1)}{\Gamma(\alpha_{2}  +(n-1)\beta_{2})\Gamma(\gamma_{2}  +(n-1) \delta_{2})}+\frac{ \Gamma(\alpha_{2})\Gamma(\gamma_{2})n(2n+1)(n+1)|B_{1}|}{\Gamma(\alpha_{2}  +(n-1)\beta_{2})\Gamma(\gamma_{2}  +(n-1) \delta_{2})}\right]\\
&=\frac{1}{6(1-|B_{1}|)}\sum_{n=1}^{\infty}  \left[\frac{\Gamma(\alpha_{1})\Gamma(\gamma_{1})(n+1)(2n+3)(n+2)}{\Gamma(\alpha_{1}  +n\beta_{1})\Gamma(\gamma_{1}  +n\delta_{1})}+\frac{ \Gamma(\alpha_{1})\Gamma(\gamma_{1})(n+1)(2n+1)n|B_{1}|}{\Gamma(\alpha_{1}  +n1\beta_{1})\Gamma(\gamma_{1}  +n \delta_{1})}\right]\\
&+\frac{1}{6(1-|B_{1}|)}\sum_{n=2}^{\infty}  \left[\frac{\Gamma(\alpha_{2})\Gamma(\gamma_{2})(n+1)(2n+1)n}{\Gamma(\alpha_{2}  +n\beta_{2})\Gamma(\gamma_{2}  +n \delta_{2})}+\frac{ \Gamma(\alpha_{2})\Gamma(\gamma_{2})(n+1)(2n+3)(n+2)|B_{1}|}{\Gamma(\alpha_{2}  +n\beta_{2})\Gamma(\gamma_{2}  +n \delta_{2})}\right].
\end{align*}
In the similar manner of proof of theorem \ref{thm9}, we obtain
\begin{align*}
&=\frac{1}{6(1-|B_{1}|}[2W_{{(\alpha_{1}, \beta_{1})},({\gamma_{1}, \delta_{1}})}^{\prime\prime\prime}(1)+9W_{{(\alpha_{1}, \beta_{1})},({\gamma_{1}, \delta_{1}})}^{\prime\prime}(1)+6(W_{{(\alpha_{2}, \beta_{2})},({\gamma_{2}, \delta_{2}})}^{\prime}(1)-1)\\
&+\frac{|B_{1|}}{6(1-|B_{1}|}\left(2 W_{{(\alpha_{1}, \beta_{1})},({\gamma_{1}, \delta_{1}})}^{\prime\prime\prime}(1)]+ 3W_{{(\alpha_{1}, \beta_{1})},({\gamma_{1}, \delta_{1}})}^{\prime\prime}(1)\right)\\
&+\frac{1}{6(1-|B_{1}|}(2 W_{{(\alpha_{2}, \beta_{2})},({\gamma_{2}, \delta_{2}})}^{\prime\prime\prime}(1)+3 W_{{(\alpha_{2}, \beta_{2})},({\gamma_{2}, \delta_{2}})}^{\prime\prime}(1))\\
&+\frac{|B_{1}|}{6(1-|B_{1}|)}\left(2W_{{(\alpha_{2}, \beta_{2})},({\gamma_{2}, \delta_{2}})}^{\prime\prime\prime}(1)+9W_{{(\alpha_{2}, \beta_{2})},({\gamma_{2}, \delta_{2}})}^{\prime\prime}(1)+6(W_{{(\alpha_{2}, \beta_{2})},({\gamma_{2}, \delta_{2}})}^{\prime}(1)-1)\right).
\end{align*}
Thus, by the asserted conditions of the theorem, desires result has been established.
\end{proof}

 \section{conclusion}\label{sec6}
 In this investigation, we have constructed a class of harmonic mappings involving with the four-parameter Wright functions and derived sufficient conditions under which these mappings exhibit significant geometric properties such as harmonic starlikeness and convexity. Additionally, we have explored the interrelation among various subclasses of harmonic close-to-convex mappings. The novel results presented herein not only extend existing work in the field but also offer a basics for further study and vital role in geometric function theory.
It is very well-known that Wright function $\mathbb{W}_{\alpha, \beta}(z)$ can be derived from $\mathcal{W}_{(\alpha, \beta), (\gamma,\delta)}(z)$ after substituting $\gamma=\delta=1$. Therefore, all the results derived such as harmonic starlikeness, convexity and harmonic close-to-convexity can be established for harmonic mapping associated with Wright function \eqref{eqwright}. Subsequently, it can be noted that,  consequences for all theorems can be obtained similar as corollary \ref{cor1} and remark \ref{reamrk1}.

{}


\begin{thebibliography}{99}
\bibitem{Andrews}
G. E. Andrews, R. Askey and R. Roy, Special functions, Encyclopedia of Mathematics and its
Applications, 71, Cambridge University Press, Cambridge, 1999.
\bibitem{ahuja}
O.~P. Ahuja, Planar harmonic convolution operators generated by hypergeometric functions, Integral Transforms Spec. Funct. {\bf 18} (2007), no.~3-4, 165--177.

\bibitem{op}
O.~P. Ahuja, Connections between various subclasses of planar harmonic mappings involving hypergeometric functions, Appl. Math. Comput. {\bf 198} (2008), no.~1, 305--316.

\bibitem{aghalary}
O.~P. Ahuja, R. Aghalary and S.~B. Joshi, Harmonic univalent functions associated with $K$-uniformly starlike functions, Math. Sci. Res. J. {\bf 9} (2005), no.~1, 9--17.

\bibitem{stegun}
M. Abramowitz and I.~A. Stegun, {\it Handbook of mathematical functions with formulas, graphs, and mathematical tables}, National Bureau of Standards Applied Mathematics Series, No. 55, U. S. Government Printing Office, Washington, DC, 1964.

\bibitem{baricz}
\'A. Baricz, {\it Generalized Bessel functions of the first kind}, Lecture Notes in Mathematics, 1994, Springer, Berlin, 2010.




\bibitem{small}
J.~G. Clunie and T.~B. Sheil-Small, Harmonic univalent functions, Ann. Acad. Sci. Fenn. Ser. A I Math. {\bf 9} (1984), 3--25.

\bibitem{choquet}
G. Choquet, Sur un type de transformation analytique g\'en\'eralisant la repr\'esentation conforme et d\'efinie au moyen de fonctions harmoniques, Bull. Sci. Math. (2) {\bf 69} (1945), 156--165.

\bibitem{das}
S. Das and K. Mehrez, J. Contemp. Math. Anal. {\bf 57} (2022), no.~1, 43--58; translated from Izv. Nats. Akad. Nauk Armenii Mat. {\bf 57} (2022), no.~1, 45--63.

\bibitem{duren}
P.~L. Duren, {\it Harmonic mappings in the plane}, Cambridge Tracts in Mathematics, 156, Cambridge Univ. Press, Cambridge, 2004.


\bibitem{jmj}
J.~M. Jahangiri, Coefficient bounds and univalence criteria for harmonic functions with negative coefficients, Ann. Univ. Mariae Curie-Sk\l odowska Sect. A {\bf 52} (1998), no.~2, 57--66.


\bibitem{jahangiri}
J.~M. Jahangiri, Harmonic functions starlike in the unit disk, J. Math. Anal. Appl. {\bf 235} (1999), no.~2, 470--477.



\bibitem{kneser}
H. Kneser, Losung der aufgabe 41. Jahresber. Deutsch. Math.-Verein. 1926;35:123-4.

\bibitem{karlsson}
P.~W. Karlsson, Book Review: A treatise on generating functions, Bull. Amer. Math. Soc. (N.S.) {\bf 19} (1988), no.~1, 346--348.

\bibitem{lewy}
H. Lewy, On the non-vanishing of the Jacobian in certain one-to-one mappings, Bull. Amer. Math. Soc. {\bf 42} (1936), no.~10, 689--692.

\bibitem{luchko}
Y.~F. Luchko and R. Gorenflo, Scale-invariant solutions of a partial differential equation of fractional order, Fract. Calc. Appl. Anal. {\bf 1} (1998), no.~1, 63--78.

\bibitem{mehrez}
K. Mehrez, New integral representations for the Fox-Wright functions and its applications, J. Math. Anal. Appl. {\bf 468} (2018), no.~2, 650--673.

\bibitem{sahu}
S. Maharana and S.~K. Sahoo, Inclusion properties of planar harmonic mappings associated with the Wright function, Complex Var. Elliptic Equ. {\bf 66} (2021), no.~10, 1619--1641.

\bibitem{mainardi}
F. Mainardi, {\it Fractional calculus and waves in linear viscoelasticity}, Imp. Coll. Press, London, 2010;



\bibitem{powel}
S. Porwal, Connections between various subclasses of planar harmonic mappings involving generalized Bessel functions, Thai J. Math. {\bf 13} (2015), no.~1, 33--42.


\bibitem{podlubny}
I. Podlubny, {\it Fractional differential equations}, Mathematics in Science and Engineering, 198, Academic Press, San Diego, CA, 1999.











\bibitem{sarkar}
N. Ta\c sar, F.~M. Sakar and B.~A. Frasin, Connections between various subclasses of planar harmonic mappings involving Mittag-Leffler functions, Afr. Mat. {\bf 35} (2024), no.~2, Paper No. 33, 11 pp.




\bibitem{wright}
E.~M. Wright, On the Coefficients of Power Series Having Exponential Singularities, J. London Math. Soc. {\bf 8} (1933), no.~1, 71--79.








\bibitem{Wright}
E.~M. Wright, Corrigenda. The Asymptotic Expansion of the Generalised Hypergeometric Function, J. London Math. Soc. {\bf 27} (1952), no.~2, 256.

\bibitem{right}
E.~M. Wright, The generalized Bessel function of order greater than one, Quart. J. Math. Oxford Ser. {\bf 11} (1940), 36--48.





\bibitem{wang}
X. Wang, X.~Q. Liang and Y.~L. Zhang, Precise coefficient estimates for close-to-convex harmonic univalent mappings, J. Math. Anal. Appl. {\bf 263} (2001), no.~2, 501--509.


\end{thebibliography}
\end{document}